\documentclass[leqno,12pt]{article} 
\usepackage{amsmath, amssymb}
\usepackage{amsthm} 
\theoremstyle{plain} 
\newtheorem{theorem}{\indent\sc Theorem}[section]
\newtheorem{lemma}[theorem]{\indent\sc Lemma}

\newtheorem{proposition}[theorem]{\indent\sc Proposition}

\theoremstyle{definition} 
\newtheorem{definition}[theorem]{\indent\sc Definition}
\newtheorem{remark}[theorem]{\indent\sc Remark}


\title{Foliations induced by metallic structures}

\author{Adara M. Blaga and Antonella Nannicini}

\date{}

\begin{document}

\maketitle

\markboth{{\small\it {\hspace{4cm} Foliations induced by metallic structures}}}{\small\it{Foliations induced by metallic structures
\hspace{4cm}}}


\begin{abstract}
We give necessary and sufficient conditions for the real distributions defined by a metallic pseudo-Riemannian structure to be integrable and geodesically invariant, in terms of associated tensor fields to the metallic structures and of adapted connections. In the integrable case, we prove a Chen-type inequality for these distributions and provide conditions for a metallic map to preserve these distributions. If the structure is metallic Norden, we describe the complex metallic distributions in the same spirit.
\end{abstract}

\bigskip

\section{Introduction}
Let $M$ be a smooth manifold and let $J$ be a $(1,1)$-tensor field on $M$. If $J^2=pJ+qI$, for some $p$ and $q$ real numbers, then $J$ is called \textit{a metallic structure on $M$} and $(M,J)$ is called \textit{a metallic manifold}. If $g$ is a pseudo-Riemannian metric on $M$ such that $J$ is $g$-symmetric, then $(J,g)$ is called a \textit{metallic pseudo-Riemannian structure} on $M$.

The aim of this paper is to consider the complementary distributions associated to a metallic pseudo-Riemannian structure and study their integrability and geodesically invariance in terms of associated tensor fields to the metallic structure and of adapted connections. In this sense, we study Schouten-van Kampen, Vr\u anceanu and Vidal connections, which seem to be the most important connections for the study of foliations of a pseudo-Riemannian manifold \cite {b:f}. Moreover, for these distributions, we prove a Chen-type inequality giving a relation between the squared norm of the mean curvature and the Chen first invariant. We also prove a leaf correspondence theorem between the leaves of two metallic pseudo-Riemannian manifolds when we have a metallic map between them with certain properties.

The sign of $p^2+4q$ is very important in the study of foliations induced by metallic structures because if it is positive, then $J$ has two real eigenvalues and if it is negative, $J$ has two complex eigenvalues. In the real case, $J$ can be related to almost product structures and in the complex case, to Norden structures on $(M,g)$.
In this paper we consider both of these cases and we describe some similarities and differences between them. In particular, in the complex case, we compute the $\bar{\partial}$-operator in terms of $J$. Moreover, we construct the metallic complex cohomology and homology groups.

Remark that some properties of metallic distributions have also been studied in \cite{ozyi}.

\section{Preliminaries}

\subsection{Metallic pseudo-Riemannian structures}


\begin{definition} \cite{bn1} Let $(M,g)$ be a pseudo-Riemannian manifold and let $J$ be a metallic structure on $M$. We say that the pair $(J,g)$ is \textit{a metallic pseudo-Riemannian structure on $M$} if $J$ is $g$-symmetric. In this case, $(M,J,g)$ is called \textit{a metallic pseudo-Riemannian manifold}. If $p^2+4q<0$, then $(J,g)$ is called \textit{a metallic Norden structure} and $(M,J,g)$ is called \textit{a metallic Norden manifold}.
\end{definition}

\begin{remark} Let $(M,g)$ be a pseudo-Riemannian manifold and let $J$ be a metallic structure on $M$ such that $J^2=pJ+qI$. If we require that $J$ is $g$-skew-symmetric, then we obtain that $p=0$. Namely, if we assume $g(JX,Y)=-g(X,JY)$, for any $X$, $Y\in C^{\infty}(TM)$, then we get
$g(JX,JY)=-g(X,J^2Y)=-pg(X,JY)-qg(X,Y)=pg(JX,Y)-qg(X,Y)$. On the other hand, $g(JX,JY)=-g(J^2X,Y)=-pg(JX,Y)-qg(X,Y)$, therefore $p=0.$ In particular, for $p\neq 0$, it is not possible to define the concept of metallic Hermitian structure.
\end{remark}

\begin{definition} \cite{bn1} i) A linear connection $\nabla$ on $M$ is called $J$-{\it connection} if $J$ is covariantly constant with respect to $\nabla$,
i.e. $\nabla J=0$.

ii) A metallic pseudo-Riemannian manifold $(M,J,g)$ such that the Levi-Civita connection $\nabla$ with respect to $g$ is a $J$-connection is called \textit{a locally metallic pseudo-Riemannian manifold}.
\end{definition}

\subsection{Associated tensors to a metallic pseudo-Riemannian structure}

For a metallic pseudo-Riemannian structure $(J,g)$ on the smooth manifold $M$ with $\nabla$ the Levi-Civita connection of $g$, we in\-tro\-duce some tensor fields \cite{m} used to characterize the properties of the metallic distributions defined by $J$:
\begin{enumerate}
  \item \textit{the $J$-bracket}
  $$[X,Y]_J:=[JX,Y]+[X,JY]-J([X,Y]),$$
where $[\cdot,\cdot]$ is the Lie bracket,
$[X,Y]=\nabla_XY-\nabla_YX$
  \item \textit{the Nijenhuis tensor associated to $J$}
  $$N_J(X,Y):=J([X,Y]_J)-[JX,JY]$$
  \item \textit{the Jordan bracket associated to $J$}
  $$\{X,Y\}_J:=\{JX,Y\}+\{X,JY\}-J(\{X,Y\}),$$
where $\{\cdot,\cdot\}$ is the Jordan bracket,
$\{X,Y\}=\nabla_XY+\nabla_YX$
  \item \textit{the Jordan tensor associated to $J$}
  $$M_J(X,Y):=J(\{X,Y\}_J)-\{JX,JY\}$$
  \item \textit{the deformation tensor associated to $J$}
  $$H_J(X,Y):=(J\circ \nabla_XJ-\nabla_{JX}J)(Y)$$
which satisfies $2H_J=N_J+M_J$.
\end{enumerate}

\begin{remark} The $J$-bracket and the associated Nijenhuis tensor can be defined for any $(1,1)$-tensor field on a smooth manifold $M$, the Jordan bracket, the associated Jordan tensor and the deformation tensor can be defined for $(1,1)$-tensor fields on a pseudo-Riemannian manifold $(M,g)$.
\end{remark}

Assume that $J$ satisfies $J^2=pJ+qI$ with $p^2+4q>0$, denote by $\sigma_{\pm}:=\frac{p\pm\sqrt{p^2+4q}}{2}$ and consider the projection operators $\mathcal{P}$ and $\mathcal{P}'$ \cite{c}:
$$
\mathcal{P}:=-\frac{1}{\sqrt{p^2+4q}}J+\frac{\sigma_{+}}{\sqrt{p^2+4q}}I, \ \
\mathcal{P}':=\frac{1}{\sqrt{p^2+4q}}J-\frac{\sigma_{-}}{\sqrt{p^2+4q}}I
$$
satisfying
$$\mathcal{P}^2=\mathcal{P}, \ \ \mathcal{P}'^2=\mathcal{P}', \ \ \mathcal{P}+\mathcal{P}'=I, \ \ \mathcal{P} \circ \mathcal{P}'=0, \ \ \mathcal{P}' \circ\mathcal{P}=0.$$

From a direct computation, we get the following:
\begin{proposition}\label{p1}
For the two projection operators $\mathcal{P}$ and $\mathcal{P}'$:
\begin{enumerate}
  \item $N_\mathcal{P}=N_{\mathcal{P}'}=\frac{1}{p^2+4q}N_J$;
  \item $M_\mathcal{P}=M_{\mathcal{P}'}=\frac{1}{p^2+4q}M_J$;
  \item $H_\mathcal{P}=H_{\mathcal{P}'}=\frac{1}{p^2+4q}H_J$.
\end{enumerate}
\end{proposition}


Consider now \textit{the deformation tensors} $H$ and $H'$:
$$
H(X,Y):=\mathcal{P}'(\nabla_{\mathcal{P}X}\mathcal{P}Y)=\mathcal{P}'((\nabla_{\mathcal{P}X}\mathcal{P})Y), \ \ H'(X,Y):=\mathcal{P}(\nabla_{\mathcal{P}'X}\mathcal{P}'Y)=\mathcal{P}((\nabla_{\mathcal{P}'X}\mathcal{P}')Y)
$$
\textit{the twisting tensors} $L$ and $L'$:
$$
L(X,Y):=\frac{1}{2}[H(X,Y)-H(Y,X)], \ \ L'(X,Y):=\frac{1}{2}[H'(X,Y)-H'(Y,X)]
$$
and \textit{the extrinsic curvature tensors} $K$ and $K'$:
$$
K(X,Y):=\frac{1}{2}[H(X,Y)+H(Y,X)], \ \ K'(X,Y):=\frac{1}{2}[H'(X,Y)+H'(Y,X)],
$$
for any $X$, $Y\in C^{\infty}(TM)$.

By a direct computation we obtain:
$$H(X,Y)=\frac{1}{(p^2+4q)\sqrt{p^2+4q}}[J(\nabla_{JX}JY)-\sigma_{+}J(\nabla_{X}JY)-\sigma_{+}J(\nabla_{JX}Y)+\sigma_{+}^2J(\nabla_{X}Y)-$$$$-
\sigma_{-}\nabla_{JX}JY-q\nabla_{X}JY-q\nabla_{JX}Y+q\sigma_{+}\nabla_XY]=$$
$$=\frac{1}{(p^2+4q)\sqrt{p^2+4q}}[J(\nabla_{JX}J)-\sigma_{+}J(\nabla_{X}J)-\sigma_{-}(\nabla_{JX}J)-q(\nabla_{X}J)](Y)$$

$$H'(X,Y)=-\frac{1}{(p^2+4q)\sqrt{p^2+4q}}[J(\nabla_{JX}JY)-\sigma_{-}J(\nabla_{X}JY)-\sigma_{-}J(\nabla_{JX}Y)+\sigma_{-}^2J(\nabla_{X}Y)-$$$$-
\sigma_{+}\nabla_{JX}JY-q\nabla_{X}JY-q\nabla_{JX}Y+q\sigma_{-}\nabla_XY]=$$
$$=-\frac{1}{(p^2+4q)\sqrt{p^2+4q}}[J(\nabla_{JX}J)-\sigma_{-}J(\nabla_{X}J)-\sigma_{+}(\nabla_{JX}J)-q(\nabla_{X}J)](Y).$$

In particular, we get:
$$H(X,Y)+H'(X,Y)=\frac{1}{(p^2+4q)\sqrt{p^2+4q}}(-\sigma_{+}+\sigma_{-})[J(\nabla_{X}J)-(\nabla_{JX}J)](Y)=$$
$$=\frac{1}{p^2+4q}H_J(X,Y).$$

Moreover:
$$L=\frac{1}{2(p^2+4q)\sqrt{p^2+4q}}(\sigma_{-}N_J-J\circ N_J), \ \ L'=-\frac{1}{2(p^2+4q)\sqrt{p^2+4q}}(\sigma_{+}N_J-J\circ N_J),$$
$$K=\frac{1}{2(p^2+4q)\sqrt{p^2+4q}}(\sigma_{-}M_J-J\circ M_J), \ \ K'=-\frac{1}{2(p^2+4q)\sqrt{p^2+4q}}(\sigma_{+}M_J-J\circ M_J).$$

\section{Metallic distributions}

Let $(M,J,g)$ be a metallic pseudo-Riemannian manifold such that $J^2=pJ+qI$ with $p^2+4q>0$.
Define the complementary distributions:
\begin{equation}\label{d}
\mathcal{D}:=\ker \mathcal{P}', \ \
\mathcal{D}':=\ker \mathcal{P}
\end{equation}
which we shall call \textit{the metallic distributions} defined by the metallic structure $J$.

\begin{remark}
The distributions $\mathcal{D}$ and $\mathcal{D}'$ are $J$-invariant, and, if  $q\neq 0$, then $\mathcal{D}$ and $\mathcal{D}'$ are also $g$-orthogonal.
\end{remark}


\begin{definition}
We say that a distribution $\mathcal{D}\subset TM$ on a smooth manifold $M$ is called

i) \textit{involutive} if $X$, $Y\in \Gamma(\mathcal{D})$ implies $[X,Y]\in \Gamma(\mathcal{D})$;

ii) \textit{integrable} if for any $x\in M$, there exists a submanifold $N_x$ which admits $\mathcal{D}|_{N_x}$ as tangent bundle.
\end{definition}

According to Frobenius theorem, a distribution $\mathcal{D}$ on $M$ is involutive if and only if it is integrable. In this case, it defines a foliation whose leaves are the maximal connected submanifolds $N_x$ of $M$ which admit $\mathcal{D}|_{N_x}$ as tangent bundle.

\begin{definition}
We say that the metallic pseudo-Riemannian manifold $(M,J,g)$ is \textit{doubly foliated} if both of the distributions
$\mathcal{D}$ and $\mathcal{D}'$ given by (\ref{d}) are integrable and \textit{singly foliated}
if only one of them is integrable.
\end{definition}

\begin{remark}\label{r3}
The distribution $\mathcal{D}$ (resp. $\mathcal{D}'$) given by (\ref{d}) is integrable if and only if $(\nabla_XJ)Y-(\nabla_YJ)X=0$, for any
$X$, $Y\in \Gamma (\mathcal{D})$ (resp. $X$, $Y\in \Gamma (\mathcal{D}')$), with $\nabla$ a torsion-free linear connection on $M$. Indeed, for $X$, $Y\in \Gamma (\mathcal{D})$ we have $JX=\sigma_{-}X$, $JY=\sigma_{-}Y$ and $J(\nabla_XY-\nabla_YX)=-(\nabla_XJ)Y+(\nabla_YJ)X+\sigma_{-}(\nabla_XY-\nabla_YX)$ which implies that $[X,Y]\in \Gamma (\mathcal{D})$ if and only if $(\nabla_XJ)Y-(\nabla_YJ)X=0$.

In particular, in a locally metallic pseudo-Riemannian manifold, the two distributions $\mathcal{D}$ and $\mathcal{D}'$ given by (\ref{d}) are both integrable.
\end{remark}

\begin{proposition}
If $(M,J,g)$ is a metallic pseudo-Riemannian manifold, then the distribution $\mathcal{D}$ is integrable if and only if:
$$J\circ N_J(X,Y)=\sigma_{-}N_J(X,Y), \ \ \textit{for any} \ \ X,Y\in C^{\infty}(TM),$$
respectively, $\mathcal{D}'$ is integrable if and only if:
$$J\circ N_J(X,Y)=\sigma_{+}N_J(X,Y), \ \ \textit{for any} \ \ X,Y\in C^{\infty}(TM).$$
In particular, both $\mathcal{D}$ and $\mathcal{D}'$ are integrable if and only if $N_J=0$.
\end{proposition}

\begin{proof}
The distribution $\mathcal{D}$ is integrable if and only if
$$\mathcal{P}'([\mathcal{P}X,\mathcal{P}Y])=0,$$ for any $X$, $Y\in C^{\infty}(TM)$. Therefore, from a direct computation and using Proposition \ref{p1}, we obtain that a necessary and sufficient condition for
$\mathcal{D}$ to be integrable is:
$$0=\mathcal{P}'([\mathcal{P}X,\mathcal{P}Y])=-\mathcal{P}'(N_{\mathcal{P}}(X,Y))=-\frac{1}{p^2+4q}\mathcal{P}'(N_J(X,Y))=$$$$=
-\frac{1}{(p^2+4q)\sqrt{p^2+4q}}[J\circ N_J(X,Y)-\sigma_{-}N_J(X,Y)].$$
\end{proof}


\begin{definition}
Given a linear connection $\nabla$ on a smooth manifold $M$, we say that a distribution $\mathcal{D}\subset TM$ is $\nabla $-\textit{geodesically invariant}
if $X$, $Y\in \Gamma (\mathcal{D})$ implies $\nabla_XY+\nabla_YX\in \Gamma (\mathcal{D})$.

In particular, if $\nabla$ is the Levi-Civita of the pseudo-Riemannian manifold $(M,g)$, then $\mathcal{D}$ is \textit{geodesically invariant}.
\end{definition}

Remark that the above condition is equivalent to the following: the distribution $\mathcal{D}$ is $\nabla $-geodesically invariant if $X\in \Gamma (\mathcal{D})$ implies $\nabla_XX\in \Gamma (\mathcal{D})$.

\begin{remark}\label{r1}
For a linear connection $\nabla$ on $M$, the distribution $\mathcal{D}$ (resp. $\mathcal{D}'$) given by (\ref{d}) is $\nabla$-geodesically invariant if and only if $(\nabla_XJ)Y+(\nabla_YJ)X=0$, for any
$X$, $Y\in \Gamma (\mathcal{D})$ (resp. $X$, $Y\in \Gamma (\mathcal{D}')$). Indeed, for $X$, $Y\in \Gamma (\mathcal{D})$ we have $JX=\sigma_{-}X$, $JY=\sigma_{-}Y$ and $J(\nabla_XY+\nabla_YX)=-(\nabla_XJ)Y-(\nabla_YJ)X+\sigma_{-}(\nabla_XY+\nabla_YX)$ which implies that $\nabla_XY+\nabla_YX\in \Gamma (\mathcal{D})$ if and only if $(\nabla_XJ)Y+(\nabla_YJ)X=0$.

In particular, for any $J$-connection $\nabla$, the distributions $\mathcal{D}$ and $\mathcal{D}'$ are $\nabla $-geodesically invariant.
\end{remark}





\begin{proposition}
If $(M,J,g)$ is a metallic pseudo-Riemannian manifold, then the distribution $\mathcal{D}$ is geodesically invariant if and only if:
$$J\circ M_J(X,Y)=\sigma_{-}M_J(X,Y), \ \ \textit{for any} \ \ X,Y\in C^{\infty}(TM),$$
respectively, $\mathcal{D}'$ is geodesically invariant if and only if:
$$J\circ M_J(X,Y)=\sigma_{+}M_J(X,Y), \ \ \textit{for any} \ \ X,Y\in C^{\infty}(TM).$$
In particular, both $\mathcal{D}$ and $\mathcal{D}'$ are geodesically invariant if and only if $M_J=0$.
\end{proposition}
\begin{proof}
The distribution $\mathcal{D}$ is geodesically invariant if and only if
$$\mathcal{P}'(\{\mathcal{P}X,\mathcal{P}Y\})=0,$$ for any $X$, $Y\in C^{\infty}(TM)$. Therefore, from a direct computation and using Proposition \ref{p1}, we obtain that a necessary and sufficient condition for
$\mathcal{D}$ to be geodesically invariant is:
$$0=\mathcal{P}'(\{\mathcal{P}X,\mathcal{P}Y\})=-\mathcal{P}'(M_{\mathcal{P}}(X,Y))=-\frac{1}{p^2+4q}\mathcal{P}'(M_J(X,Y))=$$$$=
-\frac{1}{(p^2+4q)\sqrt{p^2+4q}}[J\circ M_J(X,Y)-\sigma_{-}M_J(X,Y)].$$
\end{proof}

\begin{remark}
$J_p:=\mathcal{P}-\mathcal{P}'$ is an almost product structure on $M$ and
$$J_pX=-\frac{1}{\sqrt{p^2+4q}}(2J-pI)X,$$
for any $X\in C^{\infty}(TM)$.
\end{remark}

Direct computations provide the following relationship between $J$ and $J_p$-brackets, $J$ and $J_p$ Nijenhuis tensors, Jordan bracket and Jordan tensors of the two structures. Precisely, we have the following:
\begin{proposition}
$$[X,Y]_J=-\frac{{\sqrt{p^2+4q}}}{2}[X,Y]_{J_p}+\frac{p}{2}[X,Y]$$
$$N_J(X,Y)=\frac{p^2+4q}{4}N_{J_p}(X,Y)$$
$$\{X,Y\}_J=-\frac{{\sqrt{p^2+4q}}}{2}\{X,Y\}_{J_p}+\frac{p}{2}\{X,Y\}$$
$$M_J(X,Y)=\frac{p^2+4q}{4}M_{J_p}(X,Y).$$

In particular, the deformation tensors are related as follows:
$$H_J(X,Y)=\frac{p^2+4q}{4}H_{J_p}(X,Y).$$
\end{proposition}

The product conjugate connection of a linear connection $\nabla $ is \cite{blcr}:
\begin{equation}\label{co}
\nabla ^{({J_p})}_XY=\mathcal{P}(\nabla _X\mathcal{P}Y)-\mathcal{P}(\nabla _X\mathcal{P}'Y)-\mathcal{P}'(\nabla
_X\mathcal{P}Y)+\mathcal{P}'(\nabla _X\mathcal{P}'Y)
\end{equation}
and we have:

\begin{proposition} \cite{blcr}
If $\nabla ^{({J_p})}$ is torsion-free, then
${J_p}$ is integrable, which means that ${\mathcal D}$ and ${\mathcal
D}'$ are integrable distributions.
\end{proposition}

\begin{definition}
We say that a linear connection $\nabla$ restricts to a distribution $\mathcal{D}\subset TM$ on a metallic pseudo-Riemannian manifold $(M,J,g)$
if $Y\in \Gamma (\mathcal{D})$ implies $\nabla_XY \in \Gamma (\mathcal{D})$, for any $X\in C^{\infty}(TM)$.
\end{definition}

We have: \\
1) $\nabla $ restricts to ${\mathcal D}$ means $\mathcal{P}'(\nabla
_X\mathcal{P}Y)=0$ and $\mathcal{P}(\nabla _X\mathcal{P}Y)=\nabla _X\mathcal{P}Y$, \\
2) $\nabla $ restricts to ${\mathcal D}'$ means $\mathcal{P}(\nabla _X\mathcal{P}'Y)=0$
and $\mathcal{P}'(\nabla _X\mathcal{P}'Y)=\nabla _X\mathcal{P}'Y$.

A straightforward computation gives that the product conjugate connection $\nabla ^{({J_p})}$
defined by (\ref{co}) restricts to ${\mathcal D}$ and ${\mathcal D }'$.
Moreover, if $\nabla $ restricts to both ${\mathcal D}$ and ${\mathcal D }'$, then
\begin{equation}\label{cc}
\nabla ^{({J_p})}_XY=\nabla _X\mathcal{P}Y+\nabla _X\mathcal{P}'Y=\nabla _XY
\end{equation}
and so $\nabla$ is an ${J_p}$-connection. Let us remark that the above
connection (\ref{cc}) is exactly the Schouten-van Kampen connection of the
pair $(\mathcal{D}, \mathcal{D}')$:
$$
\nabla _XY=\mathcal{P}(\nabla _X\mathcal{P}Y)+\mathcal{P}'(\nabla _X\mathcal{P}'Y)
$$
which coincides with
the metallic natural connection $\tilde{\nabla}$ if $\nabla$ is the Levi-Civita connection of $g$.

Now we can express the Kirichenko tensor fields \cite{v:k} in
terms of the projectors $\mathcal{P}$, $\mathcal{P}'$:

\begin{proposition} \cite{blcr}
The structural and virtual tensor
fields of ${J_p}=\mathcal{P}-\mathcal{P}'$ are:
$$
\left\{
  \begin{array}{ll}
    C^{\mathcal{P}-\mathcal{P}'}_{\nabla }(X,Y)=2[\mathcal{P}(\nabla _{\mathcal{P}'X}\mathcal{P}'Y)+\mathcal{P}'(\nabla _{\mathcal{P}X}\mathcal{P}Y)] \\
    B^{\mathcal{P}-\mathcal{P}'}_{\nabla }(X,Y)=-2[\mathcal{P}(\nabla _{\mathcal{P}X}\mathcal{P}'Y)+\mathcal{P}'(\nabla _{\mathcal{P}'X}\mathcal{P}Y)].
  \end{array}
\right.
$$
\end{proposition}

Let us recall the well-known {\it fundamental tensor fields} of O'Neill-Gray:
$$
\left\{
  \begin{array}{ll}
    T(X, Y)=\mathcal{P}(\nabla _{\mathcal{P}'X}\mathcal{P}'Y)+\mathcal{P}'(\nabla _{\mathcal{P}'X}\mathcal{P}Y) \\
    A(X, Y)=\mathcal{P}'(\nabla _{\mathcal{P}X}\mathcal{P}Y)+\mathcal{P}(\nabla _{\mathcal{P}X}\mathcal{P}'Y).
  \end{array}
\right.
$$

Then, a comparison of last two equations yields
$$
\left\{
  \begin{array}{ll}
    C^{\mathcal{P}-\mathcal{P}'}_{\nabla }(X, Y)=2[T(X, \mathcal{P}'Y)+A(X, \mathcal{P}Y)] \\
    B^{\mathcal{P}-\mathcal{P}'}_{\nabla }(X, Y)=-2[T(X, \mathcal{P}Y)+A(X, \mathcal{P}'Y)]
  \end{array}
\right.
$$
a fact which justifies the second name of $T$ and $A$ as {\it
invariants} of the decomposition $TM=\mathcal{D}\oplus \mathcal{D}'$ \cite{f:ip}.

\bigskip

On $\mathcal{D}$ with the induced metric $g_{\mathcal{D}}$, we consider the induced connection from the pseudo-Riemannian manifold $(M,g,\nabla)$ by \cite{l}:
$$
\nabla^{\mathcal{D}}: \Gamma(\mathcal{D})\times \Gamma(\mathcal{D}) \rightarrow \Gamma(\mathcal{D}), \ \
\nabla^{\mathcal{D}}_XY:=\mathcal{P}(\nabla_XY)
$$
which is metric w.r.t. $g_{\mathcal{D}}$ and torsion-free w.r.t. the bracket
$$[\cdot,\cdot]_{\mathcal{D}}:\Gamma(\mathcal{D})\times \Gamma(\mathcal{D}) \rightarrow \Gamma(\mathcal{D}), \ \
[X,Y]_{\mathcal{D}}:=\mathcal{P}([X,Y]).$$
The bracket $[\cdot,\cdot]_{\mathcal{D}}$ has the usual properties of a Lie bracket excepting Jacobi identity which is satisfied if and only if $\mathcal{D}$ is integrable.

The integrability of $\mathcal{D}$ can also be characterized in terms of second fundamental form of $\mathcal{D}$:
$$h: \Gamma(\mathcal{D})\times \Gamma(\mathcal{D}) \rightarrow \Gamma(\mathcal{D}'), \ \
h(X,Y):=\nabla_XY-\nabla^{\mathcal{D}}_XY,$$
and we can state:
\begin{proposition} \cite{l}
The distribution $\mathcal{D}$ is integrable if and only if one of the following assertions holds:
i) $\nabla^{\mathcal{D}}$ is torsion-free; ii) $h$ is symmetric.
\end{proposition}

Similarly, on $(\mathcal{D}',g_{\mathcal{D}'})$ we define the induced connection from $(M,g,\nabla)$ by:
$$
\nabla^{\mathcal{D}'}: \Gamma(\mathcal{D}')\times \Gamma(\mathcal{D}') \rightarrow \Gamma(\mathcal{D}'), \ \
\nabla^{\mathcal{D}'}_XY:=\mathcal{P}'(\nabla_XY)
$$
and consider the second fundamental form $h'$ of $\mathcal{D}'$. Then the distribution $\mathcal{D}'$ is integrable if and only if one of the following assertions holds: i) $\nabla^{\mathcal{D}'}$ is torsion-free; ii) $h'$ is symmetric.

\bigskip

Remark that the restrictions of the metallic natural connection $\tilde{\nabla}$, defined in \cite{bn1}, to $\mathcal{D}$ and respectively, to $\mathcal{D}'$, coincide with the two induced connections, respectively:
$$\tilde{\nabla}|_{\Gamma(\mathcal{D})\times \Gamma(\mathcal{D})}=\nabla^{\mathcal{D}},  \ \ \tilde{\nabla}|_{\Gamma(\mathcal{D}')\times \Gamma(\mathcal{D}')}=\nabla^{\mathcal{D}'}.$$

\begin{remark}
For $p^2+4q=0$, we get only one distribution, $\ker (J - \frac{p}{2} I)$, and $J_t:= J - \frac{p}{2} I$ is an almost subtangent structure.
\end{remark}

\section{Adapted connections to $(\mathcal{D},\mathcal{D}')$}

\begin{definition}
We say that a linear connection $\nabla$ on $M$ is \textit{adapted} to the decomposition $TM=\mathcal{D}\oplus \mathcal{D}'$ if $Y\in \Gamma(\mathcal{D})$ implies $\nabla_XY\in \Gamma(\mathcal{D})$, for any $X\in C^{\infty}(TM)$ and $Y\in \Gamma(\mathcal{D}')$ implies $\nabla_XY\in \Gamma(\mathcal{D}')$, for any $X\in C^{\infty}(TM)$.
\end{definition}

\begin{remark} \label{r2}
If $(M,J)$ is a metallic manifold such that $J^2=pJ+qI$ with $p^2+4q>0$, then a linear connection $\nabla$ is adapted to $(\mathcal{D},\mathcal{D}')$ given by (\ref{d}) if and only if $\nabla$ is a $J$-connection. Indeed, for $Y\in \Gamma(\mathcal{D})$ we have $JY=\sigma_{-}Y$ and $(\nabla_XJ)Y=\sigma_{-}\nabla_XY-J(\nabla_XY)$, for any $X\in C^{\infty}(TM)$, which implies that $\nabla_XY\in \Gamma(\mathcal{D})$ if and only if $\nabla J=0$. Similarly we deduce the second implication.
\end{remark}

In \cite{b:f}, A. Bejancu and H. R. Farran gave the expression of all adapted connections to $(\mathcal{D}, \mathcal{D}')$, namely:
\begin{equation}\label{1000}
\nabla^*_XY= \mathcal{P}(\nabla_X\mathcal{P}Y)+\mathcal{P}'(\nabla_X\mathcal{P}'Y)+\mathcal{P}(S(X,\mathcal{P}Y))+
\mathcal{P}'(S(X,\mathcal{P}'Y)),
\end{equation}
for any $X$, $Y \in C^{\infty}(TM)$, where $\nabla$ is a linear connection and $S$ is a $(1,2)$-tensor field on $M$.

\subsection{Schouten-van Kampen connection}
An adapted connection to $(\mathcal{D}, \mathcal{D}')$ is \textit{the Schouten-van Kampen connection} $\tilde{\nabla}$ of the linear connection $\nabla$, obtained from (\ref{1000}) for $S:=0$:
\begin{equation}\tilde{\nabla}_XY:=\mathcal{P}(\nabla_X\mathcal{P}Y)+\mathcal{P}'(\nabla_X\mathcal{P}'Y)=\end{equation}
$$={\nabla}_XY+\mathcal{P}((\nabla_X\mathcal{P})Y)+\mathcal{P}'((\nabla_X\mathcal{P}')Y).$$

If $(M,J,g)$ is a metallic pseudo-Riemannian manifold such that $J^2=pJ+qI$ with $p^2+4q>0$ and $\nabla$ is torsion-free, then $\tilde{\nabla}$ is explicitly given by:
\begin{equation}\label{c}
\tilde{\nabla}_XY=\frac{1}{p^2+4q}[(2J-pI)(\nabla_XJY)-(pJ-(p^2+2q)I)(\nabla_XY)],
\end{equation}
for any $X$, $Y\in C^{\infty}(TM)$. Remark that if $\nabla$ is the Levi-Civita connection associated to $g$, then $\tilde{\nabla}$ is exactly the metallic natural connection defined in \cite{bn1}. Moreover, $\tilde{\nabla}$ is a metric $J$-connection, i.e. $\tilde{\nabla}g=\tilde{\nabla}J=0$, whose torsion is given by:
$$T^{\tilde{\nabla}}(X,Y)=\frac{1}{p^2+4q}[(2J-pI)({\nabla}_X JY-{\nabla}_Y JX)-(pJ+2qI)(\nabla_XY-\nabla_YX)],$$
for any $X$, $Y \in C^{\infty}(TM)$.

\subsection{Vr\u anceanu connection}

Another adapted connection to $(\mathcal{D}, \mathcal{D}')$ is \textit{Vr\u anceanu connection} $\bar{\nabla}$ of the linear connection $\nabla$,
obtained from (\ref{1000}) for
$$S(X,Y):=-\mathcal{P}(\nabla_{\mathcal{P}'X}\mathcal{P}Y)-\mathcal{P}'(\nabla_{\mathcal{P}'X}\mathcal{P}'Y)
+\mathcal{P}([\mathcal{P}'X,\mathcal{P}Y])+\mathcal{P}'([\mathcal{P}X,\mathcal{P}'Y]).$$

If $(M,J,g)$ is a metallic pseudo-Riemannian manifold such that $J^2=pJ+qI$ with $p^2+4q>0$, then $\bar{\nabla}$ is explicitly given by:
\begin{equation}\label{vr}
\bar{\nabla}_XY=\tilde{\nabla}_{\mathcal{P}X}Y+\mathcal{P}([\mathcal{P}'X,\mathcal{P}Y])+\mathcal{P}'([\mathcal{P}X,\mathcal{P}'Y])=\end{equation}
$$={\nabla}_XY+\frac{1}{p^2+4q}[2J((\nabla_{X}J)Y)-p(\nabla_{X}J)Y+J((\nabla_{Y}J)X)+(\nabla_{JY}J)X-p({\nabla}_YJ)X]+
$$$$+\frac{1}{p^2+4q}[T^{\nabla}(JX,JY)+J(T^{\nabla}(JX,Y))-pT^{\nabla}(JX,Y)-J(T^{\nabla}(X,JY))-qT^{\nabla}(X,Y)],$$
for any $X$, $Y\in C^{\infty}(TM)$.

Moreover, $\bar{\nabla}$ is a $J$-connection, i.e. $\bar{\nabla}J=0$, whose torsion is given by:
$$
T^{\bar{\nabla}}(X,Y)=\frac{1}{p^2+4q}N_J(X,Y)+\mathcal{P}'(T^{\nabla}(\mathcal{P}'X,\mathcal{P}'Y))-\mathcal{P}(T^{\nabla}(\mathcal{P}X,\mathcal{P}Y)),
$$
for any $X$, $Y\in C^{\infty}(TM)$.

\subsection{Vidal connection}

Let $(M,J,g)$ be a metallic pseudo-Riemannian manifold such that $J^2=pJ+qI$ with $p^2+4q>0$ and let $\nabla$ be the Levi-Civita connection of $g$.

Another adapted connection to $(\mathcal{D}, \mathcal{D}')$ is \textit{the Vidal connection} $\tilde{\tilde{\nabla}}$ associated to $J$,
obtained from (\ref{1000}) for
$$S(X,Y):=-\mathcal{P}(\nabla_{\mathcal{P}Y}\mathcal{P}')X-\mathcal{P}'(\nabla_{\mathcal{P}'Y}\mathcal{P})X,$$
therefore:
\begin{equation}\label{v}
\tilde{\tilde{\nabla}}_XY=\tilde{\nabla}_XY-\mathcal{P}(\nabla_{\mathcal{P}Y}\mathcal{P}')X-\mathcal{P}'(\nabla_{\mathcal{P}'Y}\mathcal{P})X=\end{equation}$$=\tilde{\nabla}_XY+\frac{1}{p^2+4q}[(\nabla_{JY}J)X+J((\nabla_YJ)X)-p(\nabla_YJ)X]=$$
$$=\nabla_XY+\frac{1}{p^2+4q}[2J((\nabla_XJ)Y)-p(\nabla_XJ)Y+J((\nabla_YJ)X)+(\nabla_{JY}J)X-p(\nabla_YJ)X],$$
for any $X$, $Y\in C^{\infty}(TM)$.

Moreover, $\tilde{\tilde{\nabla}}$ is a $J$-connection, i.e. $\tilde{\tilde{\nabla}}J=0$, whose torsion is given by:
$$
T^{\tilde{\tilde{\nabla}}}(X,Y)=\frac{1}{p^2+4q}N_J(X,Y),
$$
for any $X$, $Y\in C^{\infty}(TM)$.

\begin{remark} Vr\u anceanu connection of the Levi-Civita connection coincides with the Vidal connection.
\end{remark}

Moreover, we get:
$$
({\tilde{\tilde{\nabla}}}_Xg)(Y,Z)=-\frac{1}{p^2+4q}[g((\nabla_{JY}J)X-(\nabla_YJ)JX,Z)+g((\nabla_{JZ}J)X-(\nabla_ZJ)JX,Y)]=
$$
$$=\frac{1}{p^2+4q}[g(M_J(Y,X),Z)+g(M_J(Z,X),Y)+$$$$+g((\nabla_{JX}J)Y+(\nabla_{Y}J)JX,Z)+g((\nabla_{JX}J)Z+(\nabla_{Z}J)JX,Y)],$$
for any $X$, $Y$, $Z\in C^{\infty}(TM)$.

\bigskip

Since $\tilde{\nabla}J=\bar{\nabla}J=\tilde{\tilde{\nabla}}J=0$, from Remark \ref{r1}
we deduce:
\begin{proposition}
The distributions $\mathcal{D}$ and $\mathcal{D}'$ are $\tilde{\nabla}$-geodesically invariant, $\bar{\nabla}$-geo\-de\-sically invariant and $\tilde{\tilde{\nabla}}$-geodesically invariant.
\end{proposition}

Using the Vidal connection $\tilde{\tilde{\nabla}}$, we characterize the integrability and the geodesically invariance of the metallic distributions defined by $J$ in terms of the torsion and the covariant derivative of $g$ w.r.t. to this connection.
From all the above considerations, we can state:
\begin{theorem}
If $(M,J,g)$ is a metallic pseudo-Riemannian manifold such that $J^2=pJ+qI$ with $p^2+4q>0$, then the following assertions are equivalent:

(i) the distributions $\mathcal{D}$ and $\mathcal{D}'$ are integrable;

(ii) $N_J=0$;

(iii) $L=0$ and $L'=0$;

(iv) the Vidal connection given by (\ref{v}) is torsion-free.
\end{theorem}

\begin{theorem}
If $(M,J,g)$ is a metallic pseudo-Riemannian manifold such that $J^2=pJ+qI$ with $p^2+4q>0$, then the following assertions are equivalent:

(i) the distributions $\mathcal{D}$ and $\mathcal{D}'$ are geodesically invariant;

(ii) $M_J=0$;

(iii) $K=0$ and $K'=0$;

(iv) the Vidal connection given by (\ref{v}) is metric with respect to $g$.
\end{theorem}

\subsection{Leaves correspondence via metallic maps}


We shall provide the condition for a metallic map between two metallic pseudo-Riemannian manifolds to preserve the metallic distributions. We recall the following:

\begin{definition}
A smooth map $\Phi:(M_1,J_1)\rightarrow (M_2,J_2)$ between two metallic manifolds is called a \textit{metallic map} if:
$$
\Phi_{*}\circ J_1=J_2\circ \Phi_{*}.
$$
\end{definition}

\begin{remark}
If $\Phi:(M_1,J_1)\rightarrow (M_2,J_2)$ is a metallic map and $J_i^2=p_iJ_i+q_iI$ with $p_i$ and $q_i$ real numbers, $i=1,2$, then:

i) $\Phi_{*}\circ J_1^{2k+1}=J_2^{2k+1}\circ \Phi_{*}$, for any $k\in \mathbb{N}$;

ii) $([(p_2^2+q_2)-(p_1^2+q_1)]J_1+(p_2q_2-p_1q_1)I)(TM_1)\subset \ker \Phi_*$;

iii) in the particular case when one the structure is product and the other one is complex, then $Im J_1\subset \ker \Phi_*$.
\end{remark}


Consider a metallic map $\Phi:(M_1,J_1)\rightarrow (M_2,J_2)$ between the metallic manifolds $(M_i,J_i)$ such that $J_i^2=p_iJ_i+q_iI$ with $p_i^2+4q_i>0$, $i=1,2$, and assume that the distributions $\mathcal{D}_i$ and $\mathcal{D}_i'$, $i=1,2$, are integrable. Then they define the foliations $\mathcal{F}_i$ and $\mathcal{F}_i'$, $i=1,2$, whose leaves are trivial metallic pseudo-Riemannian manifolds.

Denoting by $\Phi^* \mathcal{D}_2$ the pull-back distribution, i.e.:
$$(\Phi^* \mathcal{D}_2)_x:=\{X_x\in T_xM: {\Phi_*}_x(X_x)\in {\mathcal{D}_2}_{\Phi(x)}\},$$
since $\Phi$ is a metallic map, we get:
$$(\Phi^* \mathcal{D}_2)_x=\{X_x\in T_xM:(J_1-{\sigma_2}_{+}I)(X_x)\in \ker {\Phi_*}_x\},$$
where ${\sigma_i}_{+}=\frac{p_i+\sqrt{p_i^2+4q_i}}{2}$, $i=1,2$ and
$$(\Phi^* \mathcal{D}_2')_x=\{X_x\in T_xM:(J_1-{\sigma_2}_{-}I)(X_x)\in \ker {\Phi_*}_x\},$$
where ${\sigma_i}_{-}=\frac{p_i-\sqrt{p_i^2+4q_i}}{2}$, $i=1,2$.

From the above considerations, we obtain a sufficient condition for the pull-back distribution $\Phi^* \mathcal{D}_2$ to coincide with one of the distributions $\mathcal{D}_1$ or $\mathcal{D}_1'$:

\begin{proposition}
If $\ker \Phi_*=(J_1-{\sigma_2}_{+}I)(\ker (J_1-{\sigma_1}_{+}I))$, then $\Phi^* \mathcal{D}_2=\mathcal{D}_1$. Moreover, if $\Phi$ is a surjective submersion with connected fibers, then a leaf of $\mathcal{F}_2$ corresponds to a leaf of $\mathcal{F}_1$.
\end{proposition}



\section{A Chen-type inequality for the metallic distributions}

A fundamental problem in the theory of submanifolds is the problem posed by B. Y. Chen \cite{chen3}, namely, to find relations between the main intrinsic and extrinsic invariants of a submanifold. In this sense, the Chen's inequalities for submanifolds in real space forms was proved by B. Y. Chen \cite{chen3}, in complex space forms by Y. Do\u{g}ru \cite{compl}, in quaternionic space forms by G. E. V\^ ilcu \cite{quat} etc.
In the same spirit, we shall prove a Chen-type inequality in the metallic case, for an integrable distribution defined by the metallic structure.

\medskip

Let $(M,J,g)$ be an $m$-dimensional metallic Riemannian manifold and assume that the distribution $\mathcal{D}$ is integrable.
In this case, the Riemann curvature tensors of $\mathcal{D}$ (computed with respect to the induced connection $\nabla^{\mathcal{D}}$ on $\mathcal{D}$ and the Lie bracket $[\cdot,\cdot]_{\mathcal{D}}$) and $M$ satisfy \cite{l}:
\begin{equation}\label{e8}
R^{\mathcal{D}}(X,Y,Z,W)=R^{M}(X,Y,Z,W)-g(h(X,Z),h(Y,W))+g(h(X,W),h(Y,Z)),
\end{equation}
for any $X$, $Y$, $Z$, $W\in \Gamma(\mathcal{D})$.

The relation between the mean curvature (the main extrinsic invariant) and the Chen first invariant (an intrinsic invariant), in a particular case of constant $J$-sectional curvature, is given in the followings.

From a direct computation we obtain:
\begin{proposition}
Let $(M,J,g)$ be an $m$-dimensional metallic Riemannian manifold ($m>2$) such that $J^2=pJ+qI$ with $p^2+4q>0$, whose Riemann curvature tensor is given by
\begin{equation} \label{e3}
R^{M}(X,Y,Z,W)=c[g(X,FW)g(Y,FZ)-g(X,FZ)g(Y,FW)],
\end{equation}
for any $X$, $Y$, $Z$, $W\in C^{\infty}(TM)$, where $F:=aJ+bI$ with $a$ and $b$ real numbers satisfying $qa^2-pab-b^2=1$.
Then the $J$-sectional curvature of $M$ is constant equal to $c$.
\end{proposition}


Denote by $H:=\frac{1}{n} tr(h)$ the \textit{mean curvature} and by $\delta_{\mathcal{D}}:=\tau^{\mathcal{D}}-\inf K^{\mathcal{D}}$ the \textit{Chen first invariant} of $\mathcal{D}$, where $\tau^{\mathcal{D}}$ denotes the scalar curvature of $\mathcal{D}$ and $K^{\mathcal{D}}$ its sectional curvature.

\begin{theorem}\label{te}
Let $(M,J,g)$ be an $m$-dimensional metallic Riemannian manifold ($m>2$) such that $J^2=pJ+qI$ with $p^2+4q>0$, whose Riemann curvature tensor is given by (\ref{e3}) and let $\mathcal{D}$ given by (\ref{d}) be an $n$-dimensional integrable distribution.
Then:
$$\delta_{\mathcal{D}}\leq \frac{c(a\sigma_{-}+b)^2(n^2-n+2)}{2}+\frac{n^2(n-2)}{2(n-1)}||H||^2.
$$
\end{theorem}
\begin{proof}
Consider an orthonormal frame field
$\{e_1,\dots, e_n\}$ for $\mathcal{D}$, $\{f_1,\dots,f_{m-n}\}$ an orthonormal frame field for $\mathcal{D}'$
and denote by
$$h_{ij}^k:=g(h(e_i,e_j),f_k).$$

From (\ref{e8}) and (\ref{e3}) we get
$$2\tau^{\mathcal{D}}=c(a\sigma_{-}+b)^2n(n-1)-||h||^2+n^2||H||^2.$$

Moreover
$$K^{\mathcal{D}}(e_1,e_2)=-c(a\sigma_{-}+b)^2-\sum_{k=1}^{m-n}h_{11}^kh_{22}^k+\sum_{k=1}^{m-n}(h_{12}^k)^2$$
and
$$\tau^{\mathcal{D}}-K^{\mathcal{D}}(e_1,e_2)=\frac{c(a\sigma_{-}+b)^2(n^2-n+2)}{2}+$$
$$+\sum_{k=1}^{m-n}[\sum_{3\leq i<j\leq n}(h_{ii}^kh_{jj}^k-(h_{ij}^k)^2)+\sum_{j=3}^{n}(h_{11}^k+h_{22}^k)h_{jj}^k-\sum_{j=3}^{n}((h_{1j}^k)^2+(h_{2j}^k)^2)]\leq$$
$$\leq \frac{c(a\sigma_{-}+b)^2(n^2-n+2)}{2}
+\frac{n-2}{2(n-1)}\sum_{k=1}^{m-n}(\sum_{j=1}^{n}h_{jj}^k)^2-\sum_{k=1}^{m-n}\sum_{j=3}^{n}((h_{1j}^k)^2+(h_{2j}^k)^2)=$$
$$=\frac{c(a\sigma_{-}+b)^2(n^2-n+2)}{2}+\frac{n^2(n-2)}{2(n-1)}||H||^2-\sum_{k=1}^{m-n}\sum_{j=3}^{n}((h_{1j}^k)^2+(h_{2j}^k)^2)\leq $$
$$\leq \frac{c(a\sigma_{-}+b)^2(n^2-n+2)}{2}+\frac{n^2(n-2)}{2(n-1)}||H||^2.$$
\end{proof}

\begin{remark}
If $p=0$ and $q=1$, i.e. $J$ is an almost product structure, then the inequality from Theorem \ref{te} becomes
$$\delta_{\mathcal{D}}\leq \frac{c(a-b)^2(n^2-n+2)}{2}+\frac{n^2(n-2)}{2(n-1)}||H||^2.
$$
In particular, if $a=1$ and $b=0$, i.e. $F=J$, we get
$$\delta_{\mathcal{D}}\leq \frac{c(n^2-n+2)}{2}+\frac{n^2(n-2)}{2(n-1)}||H||^2.
$$
\end{remark}

\section{Metallic Norden structures}

\subsection{Complex metallic distributions}

Let $(M,J,g)$ be a metallic Norden manifold such that $J^2=pJ+qI$ with $p^2+4q<0$ and let $T^{\mathbb{C}}M:=TM\otimes_{\mathbb{R}}\mathbb{C}$ be the complexified tangent bundle. Then we can define \textit{the complexified metallic pseudo-Riemannian structure}:
$$J^{\mathbb{C}}(X+iY):=JX+iJY,$$
$$g^{\mathbb{C}}(X_1+iY_1,X_2+iY_2):=g(X_1,X_2)-g(Y_1,Y_2)+i[g(X_1,Y_2)+g(Y_1,X_2)],$$
for any $X$, $X_1$, $X_2$, $Y$, $Y_1$, $Y_2\in C^{\infty}(TM)$.

Denote by $\sigma^{\mathbb{C}}_{\pm}:=\frac{p\pm\sqrt{p^2+4q}}{2}$ and consider the projection operators $\mathcal{P}^{\mathbb{C}}$ and $\mathcal{P}^{\mathbb{C}'}$:
$$
\mathcal{P}^{\mathbb{C}}:=-\frac{1}{\sqrt{p^2+4q}}J^{\mathbb{C}}+\frac{\sigma^{\mathbb{C}}_{+}}{\sqrt{p^2+4q}}I^{\mathbb{C}}, \ \
\mathcal{P}^{\mathbb{C}'}:=\frac{1}{\sqrt{p^2+4q}}J^{\mathbb{C}}-\frac{\sigma^{\mathbb{C}}_{-}}{\sqrt{p^2+4q}}I^{\mathbb{C}}
$$
satisfying
$${\mathcal{P}^{\mathbb{C}}}^2=\mathcal{P}^{\mathbb{C}}, \ \ {\mathcal{P}^{\mathbb{C}'}}^2=\mathcal{P}^{\mathbb{C}'}, \ \ \mathcal{P}^{\mathbb{C}}+\mathcal{P}^{\mathbb{C}'}=I^{\mathbb{C}}, \ \ \mathcal{P}^{\mathbb{C}} \circ \mathcal{P}^{\mathbb{C}'}=0, \ \ \mathcal{P}^{\mathbb{C}'}\circ  \ \mathcal{P}^{\mathbb{C}}=0$$
and define the complementary distributions:
\begin{equation}\label{dc}
\mathcal{D}^{\mathbb{C}}:=\ker \mathcal{P}^{\mathbb{C}'}, \\
\mathcal{D}^{\mathbb{C}'}:=\ker \mathcal{P}^{\mathbb{C}}
\end{equation}
which we shall call \textit{the complex metallic distributions} defined by $J$.

\begin{remark}
If $(M,J,g)$ is a metallic Norden manifold such that $J^2=pJ+qI$ with $p^2+4q<0$, then $\mathcal{D}^{\mathbb{C}}$ and $\mathcal{D}^{\mathbb{C}'}$ are $J^{\mathbb{C}}$-invariant, and, if $q\neq 0$, then $\mathcal{D}^{\mathbb{C}}$ and $\mathcal{D}^{\mathbb{C}'}$ are also $g^{\mathbb{C}}$-orthogonal.
\end{remark}

\begin{lemma}
$$\mathcal{D}^{\mathbb{C}'}=\overline{\mathcal{D}^{\mathbb{C}}}$$
\end{lemma}
\begin{proof} It follows from the following:
$$\sigma^{\mathbb{C}}_{+}=\frac{p+\sqrt{p^2+4q}}{2}=\frac{p+i\sqrt{-p^2-4q}}{2}=\overline{\frac{p-i\sqrt{-p^2-4q}}{2}}=
\overline{\frac{p-\sqrt{p^2+4q}}{2}}=\overline{\sigma^{\mathbb{C}}_{-}}.$$
\end{proof}

In particular, if $J$ is not trivial, that it admits two complex eigenvalues, or the two distributions are both different from 0, then the complexified tangent bundle splits as direct sum of two conjugate subbundles:
$$T^{\mathbb{C}}M={\mathcal{D}^{\mathbb{C}}} \oplus {\overline{\mathcal{D}^{\mathbb{C}}}}.$$

Extending the Lie bracket to:
$$[X_1+iY_1,X_2+iY_2]^{\mathbb{C}}:=[X_1,X_2]-[Y_1,Y_2]+i([X_1,Y_2]+[Y_1,X_2]),$$
for any $X_1$, $X_2$, $Y_1$, $Y_2\in C^{\infty}(TM)$, we say that:

\begin{definition}
A distribution $\mathcal{D}^{\mathbb{C}}\subset T^{\mathbb{C}}M$ is called \textit{integrable} if $X$, $Y\in \Gamma(\mathcal{D}^{\mathbb{C}})$ implies $[X,Y]^{\mathbb{C}}\in \Gamma(\mathcal{D}^{\mathbb{C}})$.
\end{definition}

\begin{lemma}
The distribution $\mathcal{D}^{\mathbb{C}}$ is integrable if and only if $$\mathcal{P}^{\mathbb{C}'}([\mathcal{P}^{\mathbb{C}}X,\mathcal{P}^{\mathbb{C}}Y]^{\mathbb{C}})=0,$$ for any $X$, $Y\in C^{\infty}(T^{\mathbb{C}}M)$.
\end{lemma}

\begin{proposition}
The distribution $\mathcal{D}^{\mathbb{C}}$ (resp. $\mathcal{D}^{\mathbb{C}'}$) given by (\ref{dc}) is integrable if and only if $N_J=0.$
\end{proposition}

Extending the Levi-Civita connection $\nabla$ of $g$ to:
$$\nabla^{\mathbb{C}}_{X_1+iY_1}(X_2+iY_2):=\nabla_{X_1}X_2-\nabla_{Y_1}Y_2+i(\nabla_{X_1}Y_2+\nabla_{Y_1}X_2),$$
for any $X_1$, $X_2$, $Y_1$, $Y_2\in C^{\infty}(TM)$, we pose the following:
\begin{definition}
Given a complex linear connection $\nabla^{\mathbb{C}}$ on a smooth manifold $M$, a distribution $\mathcal{D}^{\mathbb{C}}\subset T^{\mathbb{C}}M$ is called $\nabla^{\mathbb{C}}$-\textit{geodesically invariant} if $X$, $Y\in \Gamma(\mathcal{D}^{\mathbb{C}})$ implies $\nabla^{\mathbb{C}}_{X}Y+\nabla^{\mathbb{C}}_{Y}X\in \Gamma(\mathcal{D}^{\mathbb{C}})$.

In particular, if $\nabla^{\mathbb{C}}$ is the Levi-Civita connection of the pseudo-Riemannian manifold $(M,g^{\mathbb{C}})$, then $\mathcal{D}^{\mathbb{C}}$ is called \textit{geodesically invariant}.

\end{definition}

\begin{lemma}
The distribution $\mathcal{D}^{\mathbb{C}}$ is geodesically invariant if and only if
$$\mathcal{P}^{\mathbb{C}'}(\{\mathcal{P}^{\mathbb{C}}X,\mathcal{P}^{\mathbb{C}}Y\}^{\mathbb{C}})=0,$$
for any $X$, $Y\in C^{\infty}(T^{\mathbb{C}}M)$, where $\{X,Y\}^{\mathbb{C}}:=\nabla^{\mathbb{C}}_{X}Y+\nabla^{\mathbb{C}}_{X}Y$.
\end{lemma}

\begin{proposition}
The distribution $\mathcal{D}^{\mathbb{C}}$ (resp. $\mathcal{D}^{\mathbb{C}'}$) given by (\ref{dc}) is geodesically invariant if and only if $M_J=0.$
\end{proposition}

\begin{remark}\label{r4}
For a complex linear connection $\nabla^{\mathbb{C}}$ on $M$, the distribution $\mathcal{D}^{\mathbb{C}}$ (resp. $\mathcal{D}^{\mathbb{C}'}$) given by (\ref{dc}) is $\nabla^{\mathbb{C}}$-geodesically invariant if and only if $(\nabla^{\mathbb{C}}_XJ^{\mathbb{C}})Y+(\nabla^{\mathbb{C}}_YJ^{\mathbb{C}})X=0$, for any
$X$, $Y\in \Gamma (\mathcal{D}^{\mathbb{C}})$ (resp. $X$, $Y\in \Gamma (\mathcal{D}^{\mathbb{C}'})$). Indeed, for $X$, $Y\in \Gamma (\mathcal{D}^{\mathbb{C}})$ we have $J^{\mathbb{C}}X=\sigma^{\mathbb{C}}_{-}X$, $J^{\mathbb{C}}Y=\sigma^{\mathbb{C}}_{-}Y$ and $J^{\mathbb{C}}(\nabla^{\mathbb{C}}_XY+\nabla^{\mathbb{C}}_YX)=-(\nabla^{\mathbb{C}}_XJ^{\mathbb{C}})Y-(\nabla^{\mathbb{C}}_YJ^{\mathbb{C}})X+
\sigma^{\mathbb{C}}_{-}(\nabla^{\mathbb{C}}_XY+\nabla^{\mathbb{C}}_YX)$ which implies that $\nabla^{\mathbb{C}}_XY+\nabla^{\mathbb{C}}_YX\in \Gamma (\mathcal{D}^{\mathbb{C}})$ if and only if $(\nabla^{\mathbb{C}}_XJ^{\mathbb{C}})Y+(\nabla^{\mathbb{C}}_YJ^{\mathbb{C}})X=0$.

In particular, for any $J^{\mathbb{C}}$-connection $\nabla^{\mathbb{C}}$, the distributions $\mathcal{D}^{\mathbb{C}}$ and $\mathcal{D}^{\mathbb{C}'}$ are $\nabla^{\mathbb{C}}$-geo\-de\-si\-cally invariant.
\end{remark}

\begin{remark}
$J_c:=i(\mathcal{P}^{\mathbb{C}}-\mathcal{P}^{\mathbb{C}'})$ is a Norden structure on $M$ and
$$J_cX=-\frac{1}{\sqrt{-p^2-4q}}(2J-pI)X,$$
for any $X\in C^{\infty}(TM)$.
\end{remark}

By a direct computation we get:

\begin{proposition} The Nijenhuis tensors of $J_c$ and $J$ are related as follows:
$$N_{J_c}(X,Y)=\frac{4}{{-p^2-4q}}N_J(X,Y),$$
for any $X,Y \in C^{\infty}(TM)$.
\end{proposition}

Moreover, if
$$T^{\mathbb{C}}M=T^{(1,0)}M \oplus T^{(0,1)}M$$
is the decomposition of the complexified tangent bundle into $(1,0)$ and $(0,1)$ parts, with respect to the almost complex structure $J_c$, we have:
$${\mathcal{D}^{\mathbb{C}'}}=T^{(1,0)}M$$
and
$${\mathcal{D}^{\mathbb{C}}}=T^{(0,1)}M.$$

\begin{definition}
We say that a complex linear connection $\nabla^{\mathbb{C}}$ on $M$ is \textit{adapted} to the decomposition $T^{\mathbb{C}}M=\mathcal{D}^{\mathbb{C}}\oplus \mathcal{D}^{\mathbb{C}'}$ if $Y\in \Gamma(\mathcal{D}^{\mathbb{C}})$ implies $\nabla^{\mathbb{C}}_XY\in \Gamma(\mathcal{D}^{\mathbb{C}})$, for any $X\in C^{\infty}(T^{\mathbb{C}}M)$ and $Y\in \Gamma(\mathcal{D}^{\mathbb{C}'})$ implies $\nabla_XY\in \Gamma(\mathcal{D}^{\mathbb{C}'})$, for any $X\in C^{\infty}(T^{\mathbb{C}}M)$.
\end{definition}

\begin{remark} \label{re}
If $(M,J,g)$ is a metallic Norden manifold such that $J^2=pJ+qI$ with $p^2+4q<0$, then a complex linear connection $\nabla^{\mathbb{C}}$ is adapted to $(\mathcal{D}^{\mathbb{C}},\mathcal{D}^{\mathbb{C}'})$ given by (\ref{dc}) if and only if $\nabla^{\mathbb{C}}$ is a $J^{\mathbb{C}}$-connection. Indeed, for $Y\in \Gamma(\mathcal{D}^{\mathbb{C}})$ we have $J^{\mathbb{C}}Y=\sigma_{-}^{\mathbb{C}}Y$ and $(\nabla^{\mathbb{C}}_XJ^{\mathbb{C}})Y=\sigma_{-}^{\mathbb{C}}\nabla^{\mathbb{C}}_XY-J^{\mathbb{C}}(\nabla^{\mathbb{C}}_XY)$, for any $X\in C^{\infty}(T^{\mathbb{C}}M)$, which implies that $\nabla^{\mathbb{C}}_XY\in \Gamma(\mathcal{D}^{\mathbb{C}})$ if and only if $\nabla^{\mathbb{C}} J^{\mathbb{C}}=0$. Similarly we deduce the second implication.
\end{remark}

\begin{proposition}
All adapted connections to $(\mathcal{D}^{\mathbb{C}}, \mathcal{D}^{\mathbb{C}'})$ are of the form:
\begin{equation}\label{ec}
(\nabla^{\mathbb{C}})^*_XY= \mathcal{P}^{\mathbb{C}}(\nabla^{\mathbb{C}}_X\mathcal{P}^{\mathbb{C}}Y)+\mathcal{P}^{\mathbb{C}'}(\nabla^{\mathbb{C}}_X\mathcal{P}^{\mathbb{C}'}Y)+
\mathcal{P}^{\mathbb{C}}(S(X,\mathcal{P}^{\mathbb{C}}Y))+
\mathcal{P}^{\mathbb{C}'}(S(X,\mathcal{P}^{\mathbb{C}'}Y)),
\end{equation}
for any $X$, $Y \in C^{\infty}(T^{\mathbb{C}}M)$, where $\nabla^{\mathbb{C}}$ is a complex linear connection and $S$ is a complex $(1,2)$-tensor field on $M$.
\end{proposition}
\begin{proof}
We follow the same steps like in the real case \cite{b:f}.
\end{proof}

Consider the following adapted connection to $(\mathcal{D}^{\mathbb{C}}, \mathcal{D}^{\mathbb{C}'})$:

\medskip

1) \textit{The complex Schouten-van Kampen connection} $\tilde{\nabla}^{\mathbb{C}}$ of the complex linear connection $\nabla^{\mathbb{C}}$, obtained from (\ref{ec}) for $S:=0$:
$$\tilde{\nabla}^{\mathbb{C}}_XY:=\mathcal{P}^{\mathbb{C}}(\nabla^{\mathbb{C}}_X\mathcal{P}^{\mathbb{C}}Y)+
\mathcal{P}^{\mathbb{C}'}(\nabla^{\mathbb{C}}_X\mathcal{P}^{\mathbb{C}'}Y).$$

If $(M,J,g)$ is a metallic Norden manifold such that $J^2=pJ+qI$ with $p^2+4q<0$ and $\nabla^{\mathbb{C}}$ is torsion-free, then $\tilde{\nabla}^{\mathbb{C}}$ is explicitly given by:

\begin{equation}\label{g}
\tilde{\nabla}^{\mathbb{C}}_XY=\frac{1}{p^2+4q}[(2J^{\mathbb{C}}-pI^{\mathbb{C}})(\nabla^{\mathbb{C}}_XJY)-
(pJ^{\mathbb{C}}-(p^2+2q)I^{\mathbb{C}})(\nabla^{\mathbb{C}}_XY)]=
\end{equation}
$$={\nabla}^{\mathbb{C}}_XY+\frac{1}{p^2+4q}[2J^{\mathbb{C}}({\nabla}^{\mathbb{C}}_XJ^{\mathbb{C}})-p({\nabla}^{\mathbb{C}}_XJ^{\mathbb{C}})]Y,$$
for any $X$, $Y\in C^{\infty}(T^{\mathbb{C}}M)$.

Remark that if $\nabla^{\mathbb{C}}$ is the Levi-Civita connection associated to $g^{\mathbb{C}}$, then $\tilde{\nabla}^{\mathbb{C}}$ is a metric $J^{\mathbb{C}}$-connection, i.e. $\tilde{\nabla}^{\mathbb{C}}g^{\mathbb{C}}=\tilde{\nabla}^{\mathbb{C}}J^{\mathbb{C}}=0$, whose torsion is given by:
$$T^{{\tilde{\nabla}}^{\mathbb{C}}}(X,Y)=\frac{1}{p^2+4q}[(2J^{\mathbb{C}}-pI^{\mathbb{C}})({\nabla}^{\mathbb{C}}_X JY-{\nabla}^{\mathbb{C}}_Y J^{\mathbb{C}}X)-(pJ^{\mathbb{C}}+2qI^{\mathbb{C}})(\nabla^{\mathbb{C}}_XY-\nabla^{\mathbb{C}}_YX)],$$
for any $X$, $Y \in C^{\infty}(T^{\mathbb{C}}M)$.

\medskip

2) \textit{The complex Vr\u anceanu connection} $\bar{\nabla}^{\mathbb{C}}$ of the complex linear connection $\nabla^{\mathbb{C}}$, obtained from (\ref{ec}) for
$$S(X,Y):=-\mathcal{P}^{\mathbb{C}}(\nabla^{\mathbb{C}}_{\mathcal{P}^{\mathbb{C}'}X}\mathcal{P}^{\mathbb{C}}Y)-
\mathcal{P}^{\mathbb{C}'}(\nabla^{\mathbb{C}}_{\mathcal{P}^{\mathbb{C}'}X}\mathcal{P}^{\mathbb{C}'}Y)
+\mathcal{P}^{\mathbb{C}}([\mathcal{P}^{\mathbb{C}'}X,\mathcal{P}^{\mathbb{C}}Y]^{\mathbb{C}})+
\mathcal{P}^{\mathbb{C}'}([\mathcal{P}^{\mathbb{C}}X,\mathcal{P}^{\mathbb{C}'}Y]^{\mathbb{C}}).$$

If $(M,J,g)$ is a metallic Norden manifold such that $J^2=pJ+qI$ with $p^2+4q<0$, then $\bar{\nabla}^{\mathbb{C}}$ is explicitly given by:
\begin{equation}\label{...}
\bar{\nabla}^{\mathbb{C}}_XY=\tilde{\nabla}^{\mathbb{C}}_{\mathcal{P}^{\mathbb{C}}X}Y+
\mathcal{P}^{\mathbb{C}}([\mathcal{P}^{\mathbb{C}'}X,\mathcal{P}^{\mathbb{C}}Y]^{\mathbb{C}})+
\mathcal{P}^{\mathbb{C}'}([\mathcal{P}^{\mathbb{C}}X,\mathcal{P}^{\mathbb{C}'}Y]^{\mathbb{C}}),
\end{equation}
for any $X,Y\in C^{\infty}(T^{\mathbb{C}}M)$.

Moreover, $\bar{\nabla}^{\mathbb{C}}$ is a $J^{\mathbb{C}}$-connection, i.e. $\bar{\nabla}^{\mathbb{C}}J^{\mathbb{C}}=0$, whose torsion is given by:
$$
T^{\bar{\nabla}^{\mathbb{C}}}(X,Y)=\frac{1}{p^2+4q}N_{J^{\mathbb{C}}}(X,Y)+\mathcal{P^{\mathbb{C}}}'(T^{\nabla^{\mathbb{C}}}(\mathcal{P^{\mathbb{C}}}'X,\mathcal{P^{\mathbb{C}}}'Y))-\mathcal{P^{\mathbb{C}}}(T^{\nabla^{\mathbb{C}}}(\mathcal{P^{\mathbb{C}}}X,\mathcal{P^{\mathbb{C}}}Y)),
$$
for any $X$, $Y\in C^{\infty}(T^{\mathbb{C}}M)$.

\medskip

3) \textit{The complex Vidal connection} $\tilde{\tilde{\nabla}}^{\mathbb{C}}$ associated to the metallic Norden structure $(J,g)$, obtained from (\ref{ec}) for
$$S(X,Y):=-{\mathcal{P}}^{\mathbb{C}}(\nabla_{\mathcal{P}^{\mathbb{C}}Y}{{\mathcal{P}}^{\mathbb{C}}}')X-{{\mathcal{P}}^{\mathbb{C}}}'(\nabla_{{{\mathcal{P}}^{\mathbb{C}}}'Y}\mathcal{P}^{\mathbb{C}})X,$$
therefore:
\begin{equation}\label{cv}
\tilde{\tilde{\nabla}}^{\mathbb{C}}_XY=\tilde{\nabla}^{\mathbb{C}}_XY
-{\mathcal{P}}^{\mathbb{C}}(\nabla_{\mathcal{P}^{\mathbb{C}}Y}{{\mathcal{P}}^{\mathbb{C}}}')X-{{\mathcal{P}}^{\mathbb{C}}}'(\nabla_{{{\mathcal{P}}^{\mathbb{C}}}'Y}\mathcal{P}^{\mathbb{C}})X=\end{equation}
$$={\tilde{\nabla}}^{\mathbb{C}}_XY+\frac{1}{p^2+4q}[(\nabla_{J^{\mathbb{C}}Y}J^{\mathbb{C}})X+J^{\mathbb{C}}((\nabla_YJ^{\mathbb{C}})X)-p(\nabla_YJ^{\mathbb{C}})X],$$
for any $X,Y\in C^{\infty}(T^{\mathbb{C}}M)$, where $\nabla^{\mathbb{C}}$ is the Levi-Civita connection of $g^{\mathbb{C}}$.

Moreover, $\tilde{\tilde{\nabla}}^{\mathbb{C}}$ is a $J^{\mathbb{C}}$-connection, i.e. $\tilde{\tilde{\nabla}}^{\mathbb{C}}J^{\mathbb{C}}=0$, whose torsion is given by:
$$
T^{\tilde{\tilde{\nabla}}^{\mathbb{C}}}(X,Y)=\frac{1}{p^2+4q}N_{J^{\mathbb{C}}}(X,Y),
$$
for any $X,Y\in C^{\infty}(T^{\mathbb{C}}M)$.

Moreover, we get:
$$
({\tilde{\tilde{\nabla}}}^{\mathbb{C}}_Xg^{\mathbb{C}})(Y,Z)=-\frac{1}{p^2+4q}[g^{\mathbb{C}}((\nabla^{\mathbb{C}}_{J^{\mathbb{C}}Y}J^{\mathbb{C}})X-(\nabla^{\mathbb{C}}_YJ^{\mathbb{C}})J^{\mathbb{C}}X,Z)+$$$$+g^{\mathbb{C}}((\nabla_{J^{\mathbb{C}}Z}J^{\mathbb{C}})X-(\nabla^{\mathbb{C}}_ZJ^{\mathbb{C}})J^{\mathbb{C}}X,Y)]
=$$
$$=\frac{1}{p^2+4q}[g^{\mathbb{C}}(M_{J^{\mathbb{C}}}(Y,X),Z)+g^{\mathbb{C}}(M_{J^{\mathbb{C}}}(Z,X),Y)+$$$$+g^{\mathbb{C}}((\nabla^{\mathbb{C}}_{J^{\mathbb{C}}X}J^{\mathbb{C}})Y+(\nabla^{\mathbb{C}}_{Y}J^{\mathbb{C}})J^{\mathbb{C}}X,Z)+g^{\mathbb{C}}((\nabla^{\mathbb{C}}_{J^{\mathbb{C}}X}J^{\mathbb{C}})Z+(\nabla^{\mathbb{C}}_{Z}J)J^{\mathbb{C}}X,Y)],$$
for any $X$, $Y$, $Z\in C^{\infty}(T^{\mathbb{C}}M)$.

\bigskip

Since $\tilde{\nabla}^{\mathbb{C}}J^{\mathbb{C}}=\bar{\nabla}^{\mathbb{C}}J^{\mathbb{C}}=\tilde{\tilde{\nabla}}^{\mathbb{C}}J^{\mathbb{C}}=0$, from Remark \ref{r4} we deduce:
\begin{proposition}
The distributions $\mathcal{D}^{\mathbb{C}}$ and $\mathcal{D}^{\mathbb{C}'}$ are $\tilde{\nabla}^{\mathbb{C}}$-geodesically invariant, $\bar{\nabla}^{\mathbb{C}}$-geodesically invariant and $\tilde{\tilde{\nabla}}^{\mathbb{C}}$-geodesically invariant.
\end{proposition}

From all the above considerations, we can state:
\begin{theorem}
If $(M,J,g)$ is a metallic Norden manifold such that $J^2=pJ+qI$ with $p^2+4q<0$, then the following assertions are equivalent:

(i) the distributions $\mathcal{D}^{\mathbb{C}}$ and $\mathcal{D}^{\mathbb{C}'}$ are integrable;

(ii) $(M,J_c)$ is a complex manifold;



(iii) the complex Vidal connection given by (\ref{cv}) is torsion-free.

\end{theorem}

\begin{theorem}
If $(M,J,g)$ is a metallic Norden manifold such that $J^2=pJ+qI$ with $p^2+4q<0$, then the following assertions are equivalent:

(i) the distributions $\mathcal{D}^{\mathbb{C}}$ and $\mathcal{D}^{\mathbb{C}'}$ are geodesically invariant;



(ii) the complex Vidal connection given by (\ref{cv}) is metric with respect to $g^{\mathbb{C}}$.
\end{theorem}

\subsection{The $\bar\partial$-operator of a metallic complex structure}

\begin{definition} A metallic manifold $(M,J)$ such that $J^2=pJ+qI$ with $p^2+4q<0$ and $J$ integrable is called \emph{metallic complex manifold}.
\end{definition}
Let $(M,J)$ be a metallic complex manifold and let $J_c=-\frac{1}{\sqrt{-p^2-4q}}(2J-pI)$ be the associated complex structure. Consider its dual map $J^*_c: T^*M \rightarrow T^*M$, defined by $(J^*_c\alpha)(X):=\alpha(J_cX)$, for any $\alpha \in C^{\infty}(T^*M)$ and for any $X \in C^{\infty}(TM)$.

We shall define the real differential operator $d^c$ acting on forms:
$$d^c:=J^*_c\circ d\circ J^*_c,$$
where $d$ is the real differential operator.

If  $(M,J,g)$ is an integrable metallic Norden manifold, we can consider the real codifferential operator $\delta^c$ acting on forms:
$$\delta^c:=\star \circ d^c\circ \star,$$
where $\star$ is the Hodge-star operator with respect to the metric $g$.

We obtain
$$d^c \circ d^c=0, \ \ d\circ d^c+d^c\circ d=0,$$
$$\delta^c \circ \delta^c=0, \ \ \delta\circ \delta^c+\delta^c\circ \delta=0,$$
where $\delta$ is the codifferential operator, and with respect to the scalar product $\langle \cdot,\cdot\rangle$ induced by $g$, the operators $d^c$ and $\delta^c$ are adjoint, i.e.
$$\langle d^c \alpha,\beta\rangle=\langle \alpha,\delta^c\beta\rangle,$$
for any $\alpha$, $\beta\in C^{\infty}(T^*M)$.

Remark that $J^*\circ \star=\star \circ J^*$ (and $J^*_c\circ \star=\star \circ J^*_c$) implies $\delta^c=J^*_c\circ \delta\circ J^*_c$ and
$$d^c \circ J^*_c=-J^*_c\circ d, \ \ J^*_c\circ d^c=-d \circ J^*_c,$$
$$\delta^c \circ J^*_c=-J^*_c\circ \delta, \ \ J^*_c\circ \delta^c=-\delta \circ J^*_c.$$

From the above relations, we can state:
\begin{proposition}
Let $\alpha$ be a real form on $M$.

i) If $\alpha$ is $d^c$-closed (resp. $\delta^c$-coclosed), then $J^*_c\alpha$ is closed (resp. coclosed).

ii) If $\alpha$ is closed (resp. coclosed), then $J^*_c\alpha$ is $d^c$-closed (resp. $\delta^c$-coclosed).

iii) If $\alpha$ is $J^*_c$-invariant, i.e. $J^*_c\alpha=\alpha$, then $\alpha$ is $d^c$-closed (resp. $\delta^c$-coclosed)
if and only if it is closed (resp. coclosed).
\end{proposition}

Therefore, the $d^c$-closed (resp. $\delta^c$-coclosed) forms are the $J^*_c$-invariant closed (resp. coclosed) forms. Then
$$\ker (d^c)=\ker (d) \cap \{J^*_c-\textit{invariant forms}\}, \ \ Im (d^c)=J^*_c(Im (d)),$$
$$\ker (\delta^c)=\ker (\delta) \cap \{J^*_c-\textit{invariant forms}\}, \ \ Im (\delta^c)=J^*_c(Im (\delta)).$$

Then we can consider \textit{the metallic cohomology groups}
$$H^r(M):=\ker (d^c_r)/ Im (d^c_{r-1}),$$
where $$d^c_r:C^{\infty}({\Lambda}^{r}(M)) \rightarrow C^{\infty}({\Lambda}^{r+1}(M))$$
and \textit{the metallic homology groups}
$$H_r(M):=\ker (\delta^c_r)/Im (\delta^c_{r+1}),$$
where $$\delta^c_r:C^{\infty}({\Lambda}^{r}(M)) \rightarrow C^{\infty}({\Lambda}^{r-1}(M)).$$

Now we can introduce \textit{the metallic Hodge-Laplace operator}
$$\Delta^c: C^{\infty}({\Lambda}^r(M)) \rightarrow C^{\infty}({\Lambda}^r(M)), \ \ \Delta^c:=d^c\circ \delta^c+\delta^c\circ d^c,$$
which is symmetric and self-adjoint w.r.t. $\langle \cdot,\cdot\rangle$. Remark that
$$\Delta^c =-J^*_c\circ \Delta \circ J^*_c,$$
where $\Delta=d\circ \delta+\delta\circ d$ is the Hodge-Laplace operator,
and $\Delta^c$ satisfies
$$\Delta^c \circ J^*_c=J^*_c\circ \Delta, \ \ J^*_c\circ \Delta^c=\Delta \circ J^*_c.$$

\begin{definition}
A real form $\alpha$ is called \textit{$J$-harmonic} if it belongs to the kernel of the metallic Hodge-Laplace operator, i.e. $\Delta ^c\alpha=0$.
\end{definition}

From the above relations, we get:

\begin{proposition}
Let $\alpha$ be a real form on $M$.

i) If $\alpha$ is $J$-harmonic, then $J^*_c\alpha$ is harmonic.

ii) If $\alpha$ is harmonic, then $J^*_c\alpha$ is $J$-harmonic.

iii) If $\alpha$ is $J^*_c$-invariant, i.e. $J^*_c\alpha=\alpha$, then $\alpha$ is $J$-harmonic if and only if it is harmonic.

iv) $\alpha$ is $J$-harmonic if and only if it is $d^c$-closed and $\delta^c$-coclosed.
\end{proposition}

Therefore, the $J$-harmonic forms are the $J^*_c$-invariant harmonic forms. Then
$$\ker (\Delta^c)=\ker (\Delta) \cap \{J^*_c-\textit{invariant forms}\}, \ \ Im (\Delta^c)=J^*_c(Im (\Delta)).$$

\bigskip

Let $$T^{\mathbb{C}}M=T^{(1,0)}M \oplus T^{(0,1)}M={\mathcal{D}^{\mathbb{C}'}}\oplus{\mathcal{D}^{\mathbb{C}}}$$
be the decomposition of the complexified tangent bundle into $(1,0)$ and $(0,1)$ parts, with respect to the complex structure $J_c$ or, equivalently, with respect to the distributions defined by $J$.

The $\bar\partial$-operator and $\bar{\bar{\partial}}$-operator acting on $(r,s)$-forms on $M$ are defined as follows:
$$ \bar\partial : C^{\infty}({\Lambda}^{(r,s)}(M)) \rightarrow C^{\infty}({\Lambda}^{(r,s+1)}(M)), \ \ \bar\partial :={1\over 2}(d-id^c),$$
$$ \bar{\bar{\partial}} : C^{\infty}({\Lambda}^{(r,s+1)}(M)) \rightarrow C^{\infty}({\Lambda}^{(r,s)}(M)), \ \ \bar{\bar{\partial}} :={1\over 2}(\delta-i\delta^c).$$

Remark that the integrability of $J$ (which is equivalent to the integrability of $J_c$) implies
$$\bar\partial\circ \bar\partial=0, \ \ \bar{\bar{\partial}}\circ \bar{\bar{\partial}}=0,$$
therefore we can consider \textit{the metallic complex cohomology groups}
$$H^{(r,s)}(M):=\ker (\bar{\partial}_{(r,s)})/Im (\bar{\partial}_{(r,s-1)}),$$
where $$\bar{\partial}_{(r,s)}:C^{\infty}({\Lambda}^{(r,s)}(M)) \rightarrow C^{\infty}({\Lambda}^{(r,s+1)}(M))$$
and \textit{the metallic complex homology groups}
$$H_{(r,s)}(M):=\ker (\bar{\bar{\partial}}_{(r,s)})/Im (\bar{\bar{\partial}}_{(r,s+1)}),$$
where $$\bar{\bar{\partial}}_{(r,s)}:C^{\infty}({\Lambda}^{(r,s)}(M)) \rightarrow C^{\infty}({\Lambda}^{(r,s-1)}(M)).$$

\bigskip

Now, if
$${T^*}^{\mathbb{C}}M={{\mathcal{D}^*}^{\mathbb{C}}} \oplus {\overline{{\mathcal{D}^*}^{\mathbb{C}}}}$$
is the decomposition of the complexified cotangent bundle defined by $J^*$, then we get the following:

\begin{proposition} Let $(M,J)$ be a metallic complex manifold such that $J^2=pJ+qI$ with $p^2+4q<0$. Then the $\bar\partial$-operator:
$$\bar\partial ={1\over {2(p^2+4q)}}[(p^2+4q)d+i(4J^*\circ d\circ J^*-2pd\circ J^*-2pJ^*\circ d+p^2d)]$$
is acting on $C^{\infty}({{\Lambda}^r}({\mathcal{D}^*})) \otimes C^{\infty}({\Lambda}^{s}({\overline{{\mathcal{D}^*}^{\mathbb{C}}}}))$.
\end{proposition}
\begin{proof} We have:
$$d^c= [-\frac{1}{\sqrt{-p^2-4q}}(2J^*-pI)]\circ d \circ [-\frac{1}{\sqrt{-p^2-4q}}(2J^*-pI)]=$$
$$=-{\frac{1}{{p^2+4q}}(4J^*\circ d\circ J^*-2pd\circ J^*-2pJ^*\circ d+p^2d)}.$$
Then the statement.
\end{proof}

Similarly, we prove that:
\begin{proposition} Let $(M,J,g)$ be a metallic Norden manifold such that $J^2=pJ+qI$ with $p^2+4q<0$. Then the $\bar{\bar{\partial}}$-operator:
$$\bar{\bar{\partial}} ={1\over {2(p^2+4q)}}[(p^2+4q)\delta+i(4J^*\circ \delta\circ J^*-2p\delta\circ J^*-2pJ^*\circ \delta+p^2\delta)]$$
is acting on $C^{\infty}({{\Lambda}^r}({\mathcal{D}^*})) \otimes C^{\infty}({\Lambda}^{s}({\overline{{\mathcal{D}^*}^{\mathbb{C}}}}))$.
\end{proposition}

\begin{remark}
The operators $d^c$ and $\bar{\partial}$ can be defined on metallic complex manifolds and $\delta^c$, $\Delta^c$ and $\bar{\bar{\partial}}$ only on metallic Norden manifolds.
\end{remark}

\small{

\textit{Adara M. Blaga}

\textit{Department of Mathematics}

\textit{West University of Timi\c{s}oara}

\textit{Bld. V. P\^{a}rvan nr. 4, 300223, Timi\c{s}oara, Rom\^{a}nia}

\textit{adarablaga@yahoo.com}

\bigskip

\textit{Antonella Nannicini}

\textit{Department of Mathematics and Informatics "U. Dini"}

\textit{University of Florence}

\textit{Viale Morgagni, 67/a, 50134, Firenze, Italy}

\textit{antonella.nannicini@unifi.it}
}

\begin{thebibliography}{99}

\bibitem{b:f} A. Bejancu, H. R. Farran, {\it Foliations and geometric structures},
Mathematics and Its Applications \textbf{580}, Springer, Dordrecht, 2006.



\bibitem{blcr} A. M. Blaga, M. C. Crasmareanu, \textit{The geometry of product conjugate connections}, An. Univ. Ioan Cuza din Iasi, seria Matematica, tom \textbf{LIX}, Fasc. 1, (2013), 73--84.

\bibitem{bn1} A. M. Blaga, A. Nannicini, \textit{On the geometry of metallic pseudo-Riemannian structures}, submitted.




\bibitem{chen3} B. Y. Chen, \textit{Mean curvature and shape operator of isometric immersions in real space forms}, Glasgow Math. J. \textbf{38}, (1996), 87--97.

\bibitem{compl} Y. Do\u{g}ru, \textit{Chen inequalities for submanifolds of some space forms endowed with a semi-symmetric non-metric connection}, Jordan J. Math. and Stat. (JJMS) \textbf{6}, (2013), no. 4, 313--339.

\bibitem{f:ip} M. Falcitelli, S. Ianu\c s, A. M. Pastore, {\it Riemannian submersions and related topics},
 World Scientific Publishing Co., Inc., River Edge, NJ, 2004.





\bibitem{m} M. Holm, \emph{New insights in brane and Kaluza-Klein theory through almost product structures}, arXiv:hep-th/9812168, 1988.




\bibitem{c} C.-E. Hre\c tcanu, M. Crasmareanu, \textit{Metallic structures on Riemannian manifolds}, Revista de la Uni\'{o}n Matem\'{a}tica Argentina \textbf{54}, (2013), no. 2, 15--27.


\bibitem{v:k} V. F. Kirichenko, \textit{Method of generalized Hermitian
geometry in the theory of almost contact manifold}, Itogi Nauki i
Tekhniki, Problems of geometry \textbf{18} (1986), 25--71; translated in J.
Soviet. Math. \textbf{42} (1988), no. 5, 1885--1919.

\bibitem{l} M.-C. Munoz-Lecanda, \textit{On some aspects of the geometry of non integrable distributions and applications}, arXiv:1808.06704.2018.





\bibitem{ozyi} M. \"{O}zkan, F. Yilmaz, \textit{Metallic structures on differentiable manifolds}, Journal of Science and Arts Year \textbf{18}, (2018), no. 3(44), 645--660.






\bibitem{quat} G. E. V\^ ilcu, \textit{B.-Y. Chen inequalities for slant submanifolds in quaternionic space forms}, Turk. J. Math.
\textbf{34}, (2010), 115--128.


\end{thebibliography}
\end{document}